\documentclass[11pt]{amsart}

\usepackage{amsmath,amssymb,amsthm}
\usepackage{mathrsfs}
\usepackage{enumitem}
\usepackage[colorlinks=true,linkcolor=blue,citecolor=blue]{hyperref}

\newcommand{\K}{\mathbb{K}}

\newcommand{\Z}{\mathbb{Z}}
\newcommand{\G}{\mathbb{G}}

\newcommand{\GL}{\operatorname{GL}}
\newcommand{\gr}{\operatorname{gr}}
\newcommand{\co}{\operatorname{co}}
\newcommand{\ord}{\operatorname{ord}}
\newcommand{\nord}{\operatorname{nord}}
\newcommand{\id}{\mathrm{id}}
\newcommand{\Char}{\operatorname{char}}
\newcommand{\Hom}{\operatorname{Hom}}
\newcommand{\End}{\operatorname{End}}
\newcommand{\tr}{\operatorname{tr}}

\newcommand{\frB}{\mathfrak B}
\newcommand{\N}{\mathbb{N}}

\newcommand{\Pp}{\mathcal{P}}
\newcommand{\BN}{\mathcal{B}}
\newcommand{\fA}{\mathfrak{A}}
\newcommand{\YD}[2]{{}_{#1}^{#2}\mathcal{YD}}

\theoremstyle{plain}
\newtheorem{theorem}{Theorem}[section]
\newtheorem{proposition}[theorem]{Proposition}
\newtheorem{lemma}[theorem]{Lemma}
\newtheorem{corollary}[theorem]{Corollary}

\theoremstyle{definition}
\newtheorem{remark}[theorem]{Remark}

\title[$q$-Weyl Freeness and Square-Free Pointed Hopf Algebras]{A $q$-Weyl Freeness Principle for Nichols Algebras and Pointed Hopf Algebras of Square-Free Dimension}
\author{Rongchuan Xiong}
\address{Department of Mathematics, Changzhou University, Changzhou 213164, China}
\email{rcxiong@foxmail.com}

\makeatletter
\@namedef{subjclassname@2020}{\textup{2020} Mathematics Subject Classification}
\makeatother

\date{}
\begin{document}
	
	\begin{abstract}
		Let $H$ be a pointed Hopf algebra of square-free dimension over an
		algebraically closed field of characteristic $p>0$. We prove that either
		$H$ is a group algebra or $\dim H/|\G(H)|=p$, and that in the latter
		case $H$ belongs to exactly one of two explicit families of rank-one
		pointed Hopf algebras.
		
		We develop a truncated $q$-Weyl freeness principle for finite-dimensional
		Nichols algebras of quandle type. If $V=\bigoplus_{x\in X}\K e_x$,
		$X'\subsetneq X$ is a nonempty subquandle, $V'=\bigoplus_{x\in X'}\K e_x$,
		and $s\in X\setminus X'$, then
		$\mathcal B(V)\simeq\K[e_s]/(e_s^{m_s})\otimes C_{s,X'}\otimes\mathcal B(V')$
		for some graded vector space $C_{s,X'}$, where $m_s$ is the nilpotency
		order of $e_s$; in particular,
		$(m_s)_z\,\mathcal H_{\mathcal B(V')}(z)\mid\mathcal H_{\mathcal B(V)}(z)$.
		In the non-group case, this yields a $p^2$-divisibility obstruction
		that rules out noncentral support for the infinitesimal braiding.
		Together with a graded-dual argument, the resulting rank-one reduction
		forces the diagram of $H$ to have dimension $p$.
	\end{abstract}
	\maketitle
	\section{Introduction}
	\label{sec:introduction}
	
	Let $\K$ be an algebraically closed field of characteristic $p>0$, and let
	$H$  be a  pointed Hopf algebra of square-free
	dimension. A basic problem
	in the structure theory of pointed Hopf algebras is to understand how
	the arithmetic of $\dim H$ constrains the infinitesimal data from which
	$H$ is built. The coradical filtration provides a framework for
	this question. Writing
	\[
	G=\G(H),
	\qquad
	\gr H\simeq R\#\K[G],
	\]
	one obtains a connected graded braided Hopf algebra $R$ in the category
	of Yetter--Drinfeld modules over $\K[G]$, called the diagram of $H$.
	Its degree-one part $R(1)$ is the infinitesimal braiding. The Nichols
	algebras generated by Yetter--Drinfeld submodules of $R(1)$ impose
	restrictions on the structure of $R$, and hence on $H$. This is one of
	the basic ideas of the lifting method for pointed Hopf algebras
	\cite{AS02,AndruskiewitschGrana}.
	
	Square-free dimensions are a natural setting in which to study this
	interaction. Since no nontrivial square divides $\dim H$, divisibility
	properties coming from the infinitesimal braiding can have strong
	consequences. In characteristic zero, \c{S}tefan proved, in particular,
	that pointed Hopf algebras of dimension $pq$, for distinct primes $p$
	and $q$, are semisimple and hence group algebras \cite{Stefan97}. More generally, Andruskiewitsch and Natale proved that a nonsemisimple
	finite-dimensional pointed Hopf algebra in characteristic zero has
	dimension divisible by a nontrivial square
	\cite[Proposition~1.8]{AN01}. Hence every pointed Hopf algebra of
	square-free dimension in characteristic zero is a group algebra.
	
	In positive characteristic, primitive elements can generate truncated
	polynomial Hopf algebras of dimension $p$, so square-free dimension no
	longer forces $H$ to be a group algebra; non-group examples already
	occur in dimension $pq$ \cite{Xiong2023}. Outside the square-free case, the dimension of
	the diagram depends on the arithmetic of $\dim H$. For dimension $p^2$, the connected and pointed cases were studied in
	\cite{W1,WW14}; the possible dimensions of nontrivial diagram are $p$ and
	$p^2$.
	For dimension $pq^2$, the nontrivial diagram dimensions can be $p$,
	$q$ or $pq$ \cite{XiongPQ2}. In the square-free case with
	abelian coradical, the results of \cite{Xiong2023} show that the
	nontrivial diagram has dimension $p$. For a general coradical, it
	remains open whether the nontrivial diagram necessarily has prime
	dimension.
	
	The purpose of this paper is to determine what happens when the total
	dimension itself is square-free, without assuming that the group of
	group-like elements is abelian. Set
	$d:=\dim R=\dim H/|G|$. Since $\dim H$ is square-free, both $d$ and
	$|G|$ are square-free and $(d,|G|)=1$. The coprimality of $d$ and
	$|G|$ interacts strongly with the divisibility properties of Nichols
	algebras in positive characteristic, and leads to the following
	structural result.
	
	\begin{theorem}
		\label{thm:intro-structure}
		Let $\K$ be algebraically closed of characteristic $p>0$, and  
		$H$   a pointed Hopf algebra over $\K$ of square-free dimension.
		Then either $H$ is a group algebra, or $\dim H=p|\G(H)|$.
	\end{theorem}
	
Equivalently, in the non-group case the diagram has dimension $p$. To
resolve the square-free problem, we develop in
Section~\ref{sec:nilpotency} a truncated $q$-Weyl freeness principle
with one nilpotent operator. The principle has a structural consequence
of independent interest for finite-dimensional Nichols algebras of
quandle type. In abstract form, it states that if operators $L,D$
satisfy
\[
DL-qLD=\id,\qquad L^m=0,
\]
with $[j]_q\neq0$ for $1\leq j<m$, then every subspace stable under
both $L$ and $D$ is a free module over $\K[t]/(t^m)$, where $t$ acts as
$L$.
	
For Nichols algebras, the operators are given by left multiplication by
a homogeneous generator and the associated opposite braided
derivative. Combining this freeness principle with Gra\~na's relative
factorization yields an explicit three-factor decomposition for
finite-dimensional Nichols algebras of quandle type. More precisely,
let $V=\bigoplus_{x\in X}\K e_x$ be a braided vector space of quandle
type associated with a finite quandle $X$, and suppose that
$\mathcal B(V)$ is finite-dimensional. If $X'\subsetneq X$ is a
nonempty subquandle, set $V'=\bigoplus_{x\in X'}\K e_x$. For
$s\in X\setminus X'$, let $m_s=\nord(e_s)$. Then there exists a graded
vector space $C_{s,X'}$ such that
\[
\mathcal B(V)
\simeq
\K[e_s]/(e_s^{m_s})\otimes C_{s,X'}\otimes\mathcal B(V')
\]
as graded vector spaces. Consequently,
\begin{equation}
	\label{eq:intro-local-height}
	(m_s)_z\,\mathcal H_{\mathcal B(V')}(z)
	\mid
	\mathcal H_{\mathcal B(V)}(z),
	\qquad
	m_s\dim\mathcal B(V')
	\mid
	\dim\mathcal B(V),
\end{equation}
where $(m)_z=1+z+\cdots+z^{m-1}$.
	
	In characteristic $p$, this gives a particularly useful consequence:
	if two distinct homogeneous generators $e_s,e_t$ have trivial self-braiding, then $(p)_z^2\mid\mathcal H_{\mathcal B(V)}(z)$ and
	$p^2\mid\dim\mathcal B(V)$; see
	Corollary~\ref{cor:modular-two-generators}.
	
	The differential operators used here belong to a well-established
	circle of ideas around Nichols algebras. Fang developed a
	differential-algebraic framework for Nichols algebras
	\cite{FangDifferential}, and studied related $q$-boson and quantized
	Weyl algebra constructions in \cite{FangWeyl}. His
	differential-algebra approach also contains a Taylor-type freeness
	result in the generic diagonal setting
	\cite[Section~8.2, Lemma~12]{FangDifferential}, where the relevant
	parameter is not a root of unity and the resulting factor is
	polynomial. Gra\~na's freeness theorem \cite{Grana} gives relative
	factorizations with respect to Nichols subalgebras. Lochmann used
	modified shifts built from braided differential operators and
	multiplication to derive divisibility relations for dimensions and
	Hilbert series of Nichols algebras of nonabelian group type
	\cite{Lochmann}, and related factorizations for graded traces were
	obtained by Lentner--Lochmann \cite{LentnerLochmann}.
	
	In the present argument, the decisive point is that the relevant
	kernel factor in Gra\~na's decomposition is stable under multiplication
	by $e_s$ and the corresponding opposite braided derivative. On this
	invariant subspace, the $q$-Weyl relation above, together with the
	nilpotency $e_s^{m_s}=0$, yields the finite-rank freeness responsible
	for the additional factor $(m_s)_z$ in \eqref{eq:intro-local-height}.
	
	We now indicate how this divisibility enters the square-free problem.
	For a simple Yetter--Drinfeld submodule
	$U=M(g,\rho)\subseteq R(1)$, Schur's lemma gives
	$\rho(g)=\lambda\,\id$. The coprimality $(D,|G|)=1$, together with
	Nichols-algebra freeness and divisibility, forces $\lambda=1$. It
	follows that $p\mid D$ and hence $p\nmid|G|$, so the Yetter--Drinfeld
	category over $\K[G]$ is semisimple. Further divisibility arguments
	then force $R(1)$ to be simple with one-dimensional homogeneous
	components. If its support were noncentral, it would contain two
	distinct homogeneous generators with trivial self-braiding, and
	Corollary~\ref{cor:modular-two-generators} would give $p^2\mid\dim\BN(R(1))\mid D$,
	contradicting the square-freeness of $D$. Hence $\dim R(1)=1$ and its
	braiding is trivial.
	
	A graded-dual argument, recorded in
	Lemma~\ref{lem:rank-one-p-power}, then gives $\dim R=p^a$ for some
	$a\geq1$. Since $D=\dim R$ is square-free, $a=1$, and hence
	$\dim R=p$. This proves Theorem~\ref{thm:intro-structure}.
	
	Combining this reduction with the classification in \cite{Xiong2023}
	gives the classification of pointed Hopf algebras of square-free
	dimension.
	\begin{theorem}
		\label{thm:intro-classification}
		Let $\K$ be algebraically closed of characteristic $p>0$, and let
		$H$ be a pointed Hopf algebra of square-free
		dimension $n$ over $\K$. If $p\nmid n$, then $H$ is a group algebra.
		If $p\mid n$, write $n=pm$ with $(p,m)=1$. Then $H$ is either a
		group algebra or isomorphic to exactly one of the following two
		families:
		\begin{enumerate}
			\item $\fA^1_G(g,\chi)$, where $|G|=m$, $g\in Z(G)$,
			$\chi\in\Hom(G,\K^\times)$, $\chi(g)=1$, and $\fA^1_G(g,\chi)$
			is generated by $\K[G]$ and $x$ subject to
			$hx=\chi(h)xh$ and $x^p=0$;
			\item $\fA^2_G(g,\chi)$, with the same conditions on $G$, $g$, and
			$\chi$, together with $g^{p-1}=1$ and
			$\chi^{p-1}=\varepsilon$, and defining relations
			$hx=\chi(h)xh$ and $x^p=x$.
		\end{enumerate}
		In both families, $\Delta(h)=h\otimes h$ and
		$\Delta(x)=x\otimes1+g\otimes x$. The two families are mutually
		non-isomorphic, and for $i\in\{1,2\}$,
		$\fA^i_G(g,\chi)\cong \fA^i_{G'}(g',\chi')$ if and only if there
		exists a group isomorphism $F:G\to G'$ such that $F(g)=g'$ and
		$\chi'=\chi\circ F^{-1}$.
	\end{theorem}
	
	The paper is organized as follows. Section~\ref{secPre} recalls the
	background on pointed Hopf algebras, Yetter--Drinfeld modules over
	finite groups, braided vector spaces of quandle type, Nichols algebras,
	and Gra\~na's freeness theorem. Section~\ref{sec:nilpotency} develops
	the truncated $q$-Weyl freeness principle and applies it to obtain a
	local Hilbert factorization, which yields the divisibility relations
	\eqref{eq:intro-local-height}. Section~\ref{sec:reduction} develops the
	consequences of Nichols-algebra divisibility for the infinitesimal
	braiding. The final section, Section~\ref{sec:main}, proves
	Theorem~\ref{thm:intro-structure} and combines it with the
	classification results of \cite{Xiong2023} for pointed Hopf algebras
	whose diagram has dimension $p$ to obtain
	Theorem~\ref{thm:intro-classification}.
	
	\section{Preliminaries}
	\label{secPre}
	\subsection*{Conventions}
	Let $\K$ be a field, $\Char\K$ the characteristic of $\K$, and $\Z_n$ the cyclic group of order $n$. We write $\K^\times$ for the multiplicative group of $\K$, and $Z(G)$ for the center of a group $G$. 
	Our reference for Hopf algebra theory is \cite{R11}.
	
	\subsection{Yetter--Drinfeld modules over group algebras and bosonizations}
	\label{subsec:YD-group-boson}
	Let $G$ be a finite group. A \emph{left-left Yetter--Drinfeld module} over $\K[G]$ is a $\K$-vector space $V$ equipped with a left $G$-action $\cdot\colon G\times V\to V$ and a left $G$-coaction $\delta\colon V\to \K[G]\otimes V$, written $\delta(v)=v_{(-1)}\otimes v_{(0)}$, satisfying the compatibility condition
	\[
	\delta(g\cdot v)=g v_{(-1)} g^{-1}\otimes g\cdot v_{(0)},\quad \forall\,g\in G,\;v\in V.
	\]
	Write $\YD{G}{G}$ for the category of left-left Yetter--Drinfeld modules over  $\K[G]$. This category is braided monoidal. For $V,W\in\YD{G}{G}$, the braiding is given by 
	\begin{align}\label{equbraidingYDcat}
		c_{V,W}:V\otimes W\to W\otimes V,\quad v\otimes w\mapsto v_{(-1)}\cdot w\otimes v_{(0)}.
	\end{align}
	Any object $V\in\YD{G}{G}$ admits a decomposition $V=\bigoplus_{g\in G}V_g$ with $V_g:=\{v\in V\mid \delta(v)=g\otimes v\}$.
	Write $\widehat G=\Hom(G,\K^\times)$ for the group of one-dimensional characters.
	\begin{remark}\label{rmk:dimV=1}
		Suppose that $V=\K\{v\}\in\YD{G}{G}$ is one-dimensional. Then there exist $g\in Z(G)$ and $\chi\in\widehat G$ such that
		\[
		\delta(v)=g\otimes v,\qquad h\cdot v=\chi(h)v\quad(h\in G).
		\]
		Conversely, every such pair $(g,\chi)$ defines a one-dimensional Yetter--Drinfeld module. Its braiding is
		\[
		c(v\otimes v)=\chi(g)v\otimes v.
		\]
	\end{remark}
	
	\begin{proposition}\label{prop:simple-YD-group}
		Assume that $\K$ is algebraically closed and  $G$ is a finite group. Let
		$g\in G$, and $\rho:C_G(g)\to\GL(W)$  an irreducible
		representation. Set $M(g,\rho):=\K[G]\otimes_{\K[C_G(g)]}W$. Then
		$M(g,\rho)$ is a simple object of $\YD{G}{G}$ with action and
		coaction
		\[
		h\cdot(a\otimes w)=ha\otimes w,
		\qquad
		\delta(a\otimes w)=aga^{-1}\otimes(a\otimes w).
		\]
		Its support is the conjugacy class $\mathcal O_g$, every nonzero
		homogeneous component has dimension $\dim\rho$, and
		$\dim M(g,\rho)=|\mathcal O_g|\dim\rho$. Conversely, every simple
		object of $\YD{G}{G}$ is isomorphic to $M(g,\rho)$ for some $g\in G$
		and some irreducible representation $\rho$ of $C_G(g)$.
		
		Moreover, since $g\in Z(C_G(g))$, there exists $q\in\K^\times$ such
		that $\rho(g)=q\,\id_W$. Hence the braiding on the homogeneous
		component $M(g,\rho)_g$ satisfies
		\[
		c(u\otimes v)=q\,v\otimes u,
		\qquad
		u,v\in M(g,\rho)_g.
		\]
		In particular, if $\dim\rho=1$, then all nonzero homogeneous
		components of $M(g,\rho)$ are one-dimensional.
	\end{proposition}
	
	\begin{proof}
		This is the standard description of simple Yetter--Drinfeld modules
		over a group algebra. The dimension formula follows from
		$\dim M(g,\rho)=[G:C_G(g)]\dim\rho=|\mathcal O_g|\dim\rho$.
		Since $g$ is central in $C_G(g)$, Schur's lemma gives
		$\rho(g)=q\,\id_W$ for some $q\in\K^\times$. For
		$u,v\in M(g,\rho)_g$ the Yetter--Drinfeld braiding is
		$c(u\otimes v)=g\cdot v\otimes u=q\,v\otimes u$, which proves the
		remaining assertions.
	\end{proof}

	Now we recall the bosonization (Radford biproduct) construction for braided Hopf algebras in $\YD{G}{G}$.
	Let $R$ be a braided Hopf algebra in $\YD{G}{G}$. The bosonization $R\#\K[G]$ is a Hopf algebra over $\K$.
	As a vector space, $R\#\K[G]=R\otimes\K[G]$. The multiplication and comultiplication are given by the smash product and smash coproduct formulas:
	\begin{align*}
		(r\# g)(s\# h)=r(g\cdot s)\# gh,\quad
		\Delta(r\# g)=r^{(1)}\# (r^{(2)})_{(-1)}g \otimes (r^{(2)})_{(0)}\# g.
	\end{align*}
	There are canonical Hopf algebra homomorphisms
	\[
	\iota:\K[G]\to R\#\K[G],\quad g\mapsto 1\# g,\qquad
	\pi:R\#\K[G]\to\K[G],\quad r\# g\mapsto \epsilon_R(r)g.
	\]
	The map $\iota$ is injective, $\pi$ is surjective, and $\pi\circ\iota=\id_{\K[G]}$.
	Furthermore, $R$ can be recovered as the coinvariant subalgebra
	\[
	R=(R\#\K[G])^{\co\K[G]}=\{x\in R\#\K[G]\mid (\id\otimes\pi)\Delta(x)=x\otimes 1\}.
	\]
	Conversely, suppose $A$ is a Hopf algebra equipped with a surjective bialgebra map $\pi:A\to\K[G]$ admitting a bialgebra section $\iota:\K[G]\to A$ such that $\pi\circ\iota=\id_{\K[G]}$. Then $A\cong R\#\K[G]$, where $R=A^{\co\K[G]}$ is a braided Hopf algebra in $\YD{G}{G}$. See \cite{R11} for details.

	\subsection{Pointed Hopf algebras and their diagrams}
	\label{subsec:pointed-diagrams}
	For a coalgebra $C$, denote by $\{C_n\}_{n\ge 0}$ its coradical filtration, by $\G(C)$ the set of group-like elements of $C$, and for $g,h\in\G(C)$, define
	\[
	\Pp_{g,h}(C):=\{c\in C\mid \Delta(c)=c\otimes g+h\otimes c\}.
	\]
	Set $\Pp(C):=\Pp_{1,1}(C)$, the space of primitive elements. A coalgebra $C$ is \emph{pointed} if its coradical $C_0=\K[\G(C)]$. A pointed coalgebra is connected whenever $C_0=\K$. 
	
	Let $H$ be a pointed Hopf algebra over $\K$ with coproduct $\Delta$, counit $\epsilon$, and antipode $S$. The set $\G(H)$ form a group, denoted by $G$. The coradical filtration $\{H_n\}_{n\ge0}$ is a Hopf algebra filtration of $H$, and the associated graded algebra $\gr H$ is again a Hopf algebra. By the Radford-Majid decomposition theorem,
	\begin{equation}\label{eq:diagram-decomposition}
		\gr H\cong R\#\K[G],
	\end{equation}
	where $R$ is a connected, coradically graded braided Hopf algebra in $\YD{G}{G}$, which  is called the \emph{diagram} of $H$ \cite{AS02}.
	Write $R=\bigoplus_{n\ge0}R(n)$ for the coradical grading. The component $R(1)$ is called the infinitesimal braiding of $H$.
	If $H$ is finite-dimensional, then
	\begin{equation}\label{eq:dim-factorization}
		\dim H=|G|\cdot\dim R.
	\end{equation}
	In particular, if $\dim H$ is square-free, then $|G|$ and $\dim R$ are square-free and coprime, i.e., $(|G|,\dim R)=1$. This is the starting point for our square-free reduction.
	
	\subsection{Racks, quandles, and braided vector spaces of group type}
	\label{subsec:quandles}
	A \emph{rack} is a set $X$ with a binary operation $\triangleright$ such that every left translation map $y\mapsto x\triangleright y$ is bijective and the left self-distributive law
	\[
	x\triangleright(y\triangleright z)=(x\triangleright y)\triangleright(x\triangleright z)
	\]
	holds for all $x,y,z\in X$. A rack is a \emph{quandle} if additionally $x\triangleright x=x$ for every $x\in X$. See \cite[Section~1]{AndruskiewitschGrana}.
	
	Conjugation in $G$ gives a natural quandle structure on each conjugacy class $X\subseteq G$, by setting
	\[
	x\triangleright y=xyx^{-1}.
	\]
	Let $V\in\YD{G}{G}$ and set $X=\operatorname{supp}V:=\{g\in G:V_g\neq0\}$. The Yetter--Drinfeld condition makes $X$ conjugation-stable. Assume that every nonzero homogeneous component $V_x$, $x\in X$, is one-dimensional, and pick a nonzero vector $e_x\in V_x$ for each $x\in X$. Then
	\[
	V=\bigoplus_{x\in X}\K e_x.
	\]
	Using the braiding formula \eqref{equbraidingYDcat}, we have
	\[
	c(e_x\otimes e_y)=q_{x,y}\,e_{x\triangleright y}\otimes e_x,
	\qquad
	q_{x,y}\in\K^\times.
	\]
	Braided vector spaces of this form are called of \emph{quandle type}.
	In particular,
	\[
	c(e_x\otimes e_x)=q_{x,x}e_x\otimes e_x.
	\]
	If $V=M(g,\rho)$ with $\dim\rho=1$, then
	Proposition~\ref{prop:simple-YD-group} shows that its nonzero
	homogeneous components are indexed by $\mathcal O_g$ and are
	one-dimensional, with $q_{x,x}=\rho(g)$ for $x\in\mathcal O_g$.
	Thus, if $\rho(g)=1$, all diagonal braiding scalars are equal to $1$.
	Consequently, if
	\[
	U=M(g,\rho)\subseteq R(1),
	\qquad
	\dim\rho=1,
	\]
	is a simple Yetter--Drinfeld submodule and $g\notin Z(G)$, then
	$|\mathcal O_g|>1$ and $U$ is a quandle-type braided vector space
	supported on the nontrivial conjugacy class $\mathcal O_g$.
	
	\subsection{Nichols algebras, braided derivatives, and freeness}
	\label{subsec:nichols-derivatives}
	For $V\in\YD{G}{G}$, the \emph{Nichols algebra} $\BN(V)$ is a strictly $\N$-graded Hopf algebra in $\YD{G}{G}$, generated as an algebra by $V$. By definition, $\BN(V)=\bigoplus_{n\ge0}\BN(V)(n)$ satisfies
	\[
	\BN(V)(0)=\K,\quad \BN(V)(1)=V,\quad \Pp(\BN(V))=V,\quad \BN(V)\text{ is algebraically generated by }V.
	\]
	It can be constructed as the quotient $\BN(V)=T(V)/I(V)$, where $T(V)$ is the tensor algebra of $V$, and $I(V)$ is the largest homogeneous braided Hopf ideal of $T(V)$ satisfying $I(V)\cap(\K\oplus V)=0$. The structure of $\BN(V)$ depends only on the braided vector space $(V,c)$ \cite{AS02}.  
	
	\begin{remark}\label{rmk-1-1-1-1-1}
		Let $\Char\K=p>0$. If $V=\K x$ and $c(x\otimes x)=x\otimes x$, then $\BN(V)\cong\K[x]/(x^p)$, with $x$ primitive.
	\end{remark}
	
	Assume that $V$ is of quandle type such that $\BN(V)$ is finite-dimensional. For each $x\in X$, define the nilpotency order
	\[
	m_x=\nord(e_x)=\min\{j\ge1:e_x^j=0\}\in\N_{\ge2}.
	\]
	For integer $m\ge1$, set $(m)_z=1+z+\cdots+z^{m-1}\in\Z[z]$. For any graded subspace $W\subseteq\BN(V)$, define
	\[
	\mathcal H_W(z)=\sum_{d\ge0}\dim W(d)\,z^d.
	\]

	For $i,j\ge0$, let
	\[
	\Delta_{i,j}\colon \BN(V)(i+j)\longrightarrow \BN(V)(i)\otimes \BN(V)(j)
	\]
	denote the homogeneous component of the coproduct of bidegree $(i,j)$; equivalently,
	\[
	\Delta_{i,j}=(\pi_i\otimes\pi_j)\circ \Delta|_{\BN(V)(i+j)},
	\]
	where $\pi_n\colon \BN(V)\to \BN(V)(n)$ is the canonical graded projection.
	
	The braided and opposite braided derivatives with respect to $e_x$ are the linear maps
	\[
	\partial_x,\ \partial_x^{\mathrm{op}}\colon \BN(V)\to \BN(V)
	\]
	determined by the degree-one components of the coproduct. For homogeneous $v\in \BN(V)(d)$ with $d\ge1$,
	\[
	\Delta_{d-1,1}(v)=\sum_{x\in X}\partial_x(v)\otimes e_x,\qquad
	\Delta_{1,d-1}(v)=\sum_{x\in X}e_x\otimes\partial_x^{\mathrm{op}}(v),
	\]
	and $\partial_x(1)=\partial_x^{\mathrm{op}}(1)=0$. Write $D_s:=\partial_s^{\mathrm{op}}$. For a subset $X'\subseteq X$, set
	\[
	V'=\bigoplus_{x\in X'}\K e_x,\qquad K_{X'}=\bigcap_{x\in X'}\ker\partial_x.
	\]
	
	Let $\delta_d:\BN(V)(d)\to V^{\otimes d}$ be the iterated degree-one component of the coproduct: $\delta_1=\id_V$ and $\delta_d=(\delta_{d-1}\otimes\id)\Delta_{d-1,1}$ for $d>1$.
	
	The following standard properties of braided derivatives will be used
	repeatedly; see, for instance, \cite{MilinskiSchneider},
	\cite[\S2.2]{AHS}, and \cite[\S6.1]{AndruskiewitschGrana}.
	
	\begin{lemma}[Braided-derivative identities]
		\label{lem:derivative-identities-prelim}
		For all $x,s\in X$ and all $u,v\in\BN(V)$,
		\[
		\partial_x(uv)=u\partial_x(v)+\partial_x(u)(x\cdot v),\qquad
		\partial_x\partial_s^{\mathrm{op}}=\partial_s^{\mathrm{op}}\partial_x.
		\]
		In particular, $\partial_x(e_sv)=e_s\partial_x(v)$ for $x\neq s$. Moreover the map $\delta_d:\BN(V)(d)\to V^{\otimes d}$ is injective for all $d\ge1$, so
		\[
		\bigcap_{x\in X}\ker\partial_x=\K1.
		\]
	\end{lemma}

	The following is a specialization of Gra\~na's freeness
	theorem \cite[Theorem~3.8]{Grana}; cf.~\cite{Lochmann}
	for the subrack formulation.
	\begin{theorem}\label{thm:grana-freeness-prelim}
		Let $V=\bigoplus_{x\in X}\K e_x$ be a braided vector space of quandle
		type such that $\BN(V)$ is finite-dimensional, and let
		$X'\subseteq X$ be a nonempty subquandle. Set
		\[
		V'=\bigoplus_{x\in X'}\K e_x,
		\qquad
		K_{X'}=\bigcap_{x\in X'}\ker\partial_x.
		\]
		Then multiplication induces an isomorphism of graded vector spaces
		\[
		K_{X'}\otimes\BN(V')
		\xrightarrow{\;\sim\;}
		\BN(V).
		\]
		Consequently,
		\[
		\mathcal H_{\BN(V)}(z)
		=
		\mathcal H_{K_{X'}}(z)\,
		\mathcal H_{\BN(V')}(z),
		\]
		and in particular
		\[
		\dim\BN(V')\mid\dim\BN(V).
		\]
	\end{theorem}

	\begin{remark}	\cite[Remark~3.10]{Grana} \label{rmk:grana-product-divisibility}
		Let $\K$ be algebraically closed of characteristic $p>0$ with
		$p\nmid |G|$, and let
		\[
		V=\bigoplus_{i=1}^r M(g_i,\rho_i)
		\]
		be a semisimple Yetter--Drinfeld module such that $\BN(V)$ is
		finite-dimensional. For each $i$, choose $0\neq z_i\in M(g_i,\rho_i)_{g_i}$.
		Since $g_i\in Z(C_G(g_i))$, Schur's lemma gives $\rho_i(g_i)=q_i\,\id$
		for some $q_i\in\K^\times$, and hence
		\[
		c(z_i\otimes z_i)=q_i z_i\otimes z_i.
		\]
		Set $N_i:=\dim\BN(\K z_i)=\nord(z_i)$. 
		Then the product-divisibility consequence of
		\cite[Remark~3.10]{Grana} gives
		\[
		\prod_{i=1}^r N_i\mid\dim\BN(V).
		\]
		In particular, if $q_i=1$ for every $i$, then $\BN(\K z_i)\cong\K[z_i]/(z_i^p)$ and $N_i=p$.
		Therefore $p^r\mid\dim\BN(V)$.
		
	\end{remark}

	\section{A $q$-Weyl freeness principle and local divisibility}
	\label{sec:nilpotency}
	
	In this section we prove the truncated $q$-Weyl freeness principle and
	apply it to Nichols algebras of quandle type. This leads to the local
	factorization in Theorem~\ref{thm:local-height-divisibility} and to the
	divisibility consequences used in the square-free reduction.

	For later use, define the $q$-factorials by
	\[
	[j]_q!=\prod_{r=1}^j[r]_q,
	\qquad
	[r]_q=1+q+\cdots+q^{r-1},
	\]
	with $[0]_q!=1$.
	\begin{theorem}[Truncated $q$-Weyl freeness]
		\label{thm:truncated-q-taylor}
		Let $q\in\K^\times$ and let
		\[
		\mathsf W_q:=\K\langle L,D\rangle/(DL-qLD-1)
		\]
		be the rank-one $q$-Weyl algebra. Let $W$ be a nonzero
		$\mathsf W_q$-module, and suppose that, for some $m\geq2$, the action
		of $L$ satisfies
		\[
		L^m=0,
		\qquad
		[j]_q\neq0,\quad 1\leq j<m.
		\]
		Then $[m]_q=0$. Set
		\[
		C:=D^{m-1}(\ker L),
		\qquad
		A_m:=\K[t]/(t^m),
		\]
		and regard $W$ as an $A_m$-module via $t\cdot w=Lw$. Then the map
		\[
		A_m\otimes_\K C\longrightarrow W,\qquad
		t^j\otimes c\longmapsto L^jc,
		\]
		is an isomorphism of $A_m$-modules. Equivalently,
		\[
		W=\bigoplus_{j=0}^{m-1}L^jC.
		\]
		If, moreover, $W=\bigoplus_{d\geq0}W(d)$ is graded, $L$ and $D$ act
		homogeneously in degrees $1$ and $-1$, respectively, and each $W(d)$
		is finite-dimensional, then $C$ is graded and
		\[
		\mathcal H_W(z)=(m)_z\,\mathcal H_C(z).
		\]
		In particular, if $W$ is finite-dimensional, then
		$m\mid\dim_\K W$.
	\end{theorem}
	
	\begin{proof}
		From $DL=\id+qLD$ we obtain
		\[
		LD=q^{-1}DL-q^{-1}\id.
		\]
		Induction on $r$ gives
		\begin{equation}
			\label{eq:q-weyl-LDr}
			LD^r=q^{-r}D^rL-q^{-r}[r]_qD^{r-1},\qquad r\geq1.
		\end{equation}
		Hence, for $u\in\ker L$,
		\[
		LD^ru=-q^{-r}[r]_qD^{r-1}u,
		\]
		and iteration yields
		\begin{equation}
			\label{eq:q-weyl-top-bottom}
			L^rD^ru
			=
			(-1)^rq^{-r(r+1)/2}[r]_q!\,u,\qquad 0\leq r<m.
		\end{equation}
		Since $[m-1]_q!\neq0$, taking $r=m-1$ shows that $D^{m-1}$ is
		injective on $\ker L$. Moreover, if $c=D^{m-1}u\in C$ with
		$u\in\ker L$, then
		\[
		L^{m-1}c=(-1)^{m-1}q^{-m(m-1)/2}[m-1]_q!\,u,
		\]
		so $L^{m-1}$ is injective on $C$.
		
		On the other hand, induction on $n$ in the defining relation gives
		\begin{equation}
			\label{eq:q-weyl-DLn}
			DL^n=[n]_qL^{n-1}+q^nL^nD,\qquad n\geq1.
		\end{equation}
		Since $W\neq0$ and $L^m=0$, we have $\ker L\neq0$. Hence $C\neq0$,
		because $D^{m-1}$ is injective on $\ker L$. It follows that
		$L^{m-1}\neq0$, since $L^{m-1}$ is injective on $C$. Taking $n=m$
		in \eqref{eq:q-weyl-DLn} and using $L^m=0$, we obtain
		\[
		0=DL^m=[m]_qL^{m-1}.
		\]
		Thus $[m]_q=0$.
		
		We next prove that $W=C\oplus LW$. Using \eqref{eq:q-weyl-DLn} and
		induction on $r$, we obtain
		\begin{equation}
			\label{eq:q-weyl-normal-order}
			D^rL^r=[r]_q!\,\id_W+LA_r
		\end{equation}
		for some $A_r\in\End_\K(W)$.
		
		Given $w\in W$, set
		$c=([m-1]_q!)^{-1}D^{m-1}L^{m-1}w$. Since $L^m=0$, we have
		$L^{m-1}w\in\ker L$, and hence $c\in C$. By
		\eqref{eq:q-weyl-normal-order}, $w-c\in LW$, so $W=C+LW$. If
		$c\in C\cap LW$, say $c=Lw$, then $L^{m-1}c=L^mw=0$. Since $L^{m-1}$
		is injective on $C$, we obtain $c=0$. Thus $W=C\oplus LW$.
		
		Iterating the decomposition $W=C\oplus LW$ and using $L^m=0$ gives
		$W=\sum_{j=0}^{m-1}L^jC$. This sum is direct. Indeed, suppose
		$\sum_{j=0}^{m-1}L^jc_j=0$ with $c_j\in C$. Applying $L^{m-1}$ gives
		$L^{m-1}c_0=0$, hence $c_0=0$. Applying $L^{m-2}$ to the remaining
		relation gives $L^{m-1}c_1=0$, hence $c_1=0$, and continuing in this
		way gives $c_j=0$ for all $j$. Therefore
		\[
		W=\bigoplus_{j=0}^{m-1}L^jC.
		\]
		
		Since $L^{m-1}$ is injective on $C$, so is $L^j|_C$ for every
		$0\leq j<m$. Hence the map
		\[
		A_m\otimes_\K C\longrightarrow W,\qquad
		t^j\otimes c\longmapsto L^jc,
		\]
		is an $A_m$-module isomorphism.
		
		If $W$ is graded and $L,D$ have degrees $1$ and $-1$, respectively,
		then $\ker L$ and hence $C=D^{m-1}(\ker L)$ are graded. Since $L^j|_C$
		is injective and has degree $j$,
		\[
		\mathcal H_{L^jC}(z)=z^j\mathcal H_C(z),
		\]
		and therefore
		\[
		\mathcal H_W(z)=\sum_{j=0}^{m-1}z^j\mathcal H_C(z)
		=(m)_z\,\mathcal H_C(z).
		\]
		Finally, if $W$ is finite-dimensional, then the direct-sum
		decomposition gives $\dim W=m\dim C$, and hence $m\mid\dim W$.
	\end{proof}

	We now apply Theorem \ref{thm:truncated-q-taylor} to Nichols algebras. Let
	$V=\bigoplus_{x\in X}\K e_x$ be a finite-dimensional Yetter--Drinfeld
	module of quandle type with one-dimensional homogeneous components, so
	that $c(e_x\otimes e_y)=q_{x,y}e_{x\triangleright y}\otimes e_x$.
	Assume that $\frB=\BN(V)$ is finite-dimensional. For $s\in X$, set
	$q_s=q_{s,s}$, $m_s=\nord(e_s)$, $L_s(v)=e_sv$, and
	$D_s=\partial_s^{\mathrm{op}}$.
	
	\begin{proposition}[The local $q$-Weyl pair]
		\label{pro:local-q-weyl}
		For every $s\in X$, $D_sL_s-q_sL_sD_s=\id_{\frB}$. Moreover,
		$L_s^{m_s}=0$, $[j]_{q_s}\neq0$ for $1\leq j<m_s$, and
		$[m_s]_{q_s}=0$. Thus $m_s$ is the first positive integer for which
		$[m_s]_{q_s}=0$, and the subalgebra of $\frB$ generated by $e_s$ has
		basis $1,e_s,\ldots,e_s^{m_s-1}$. In particular, it is naturally
		identified with the rank-one Nichols algebra
		$\BN(\K e_s)\cong\K[t]/(t^{m_s})$.
	\end{proposition}
	
	\begin{proof}
		Let $v\in\frB$ be homogeneous. Since
		$\Delta(e_s)=e_s\otimes1+1\otimes e_s$, the left-degree-one component
		of $\Delta(e_sv)=\Delta(e_s)\Delta(v)$ is
		\[
		e_s\otimes v+\sum_{y\in X}(s\cdot e_y)\otimes e_s\partial_y^{\mathrm{op}}(v).
		\]
		Now $s\cdot e_y=q_{s,y}e_{s\triangleright y}$. Since the left
		translation $y\mapsto s\triangleright y$ is bijective and fixes $s$,
		we have $s\triangleright y=s$ if and only if $y=s$. Extracting the
		coefficient of $e_s$ in the first tensor factor gives
		$D_s(e_sv)=v+q_se_sD_s(v)$. Thus, by linearity,
		$D_sL_s-q_sL_sD_s=\id_{\frB}$.
		
		By definition of $m_s$, $L_s^{m_s}=0$, and $L_s^j\neq0$ for
		$j<m_s$, since $L_s^j(1)=e_s^j\neq0$.
		
		Using the braided product rule,
		$\partial_x(e_s^j)=\delta_{x,s}[j]_{q_s}e_s^{j-1}$. If $1\leq j<m_s$
		and $[j]_{q_s}=0$, then the nonzero positive-degree element $e_s^j$
		lies in $\bigcap_{x\in X}\ker\partial_x=\K1$, a contradiction. Hence
		$[j]_{q_s}\neq0$ for $1\leq j<m_s$. On the other hand,
		$0=\partial_s(e_s^{m_s})=[m_s]_{q_s}e_s^{m_s-1}$ and
		$e_s^{m_s-1}\neq0$, so $[m_s]_{q_s}=0$.
		
		Finally, by the definition of $m_s$, the elements
		$1,e_s,\ldots,e_s^{m_s-1}$ are nonzero and have distinct degrees,
		hence are linearly independent; they clearly span the subalgebra
		generated by $e_s$. Thus this subalgebra is isomorphic to
		$\K[t]/(t^{m_s})$, and hence to the rank-one Nichols algebra
		$\BN(\K e_s)$.
	\end{proof}
	
	The preceding lemma immediately yields a stable-subspace version of the
	truncated $q$-Taylor decomposition.
	
	\begin{corollary}[Stable-subspace factorization]
		\label{cor:stable-q-weyl-factorization}
		Fix $s\in X$. Let $W\subseteq\frB$ be a graded subspace stable
		under both $L_s$ and $D_s$. Set
		\(
		C_{s,W}:=D_s^{m_s-1}(W\cap\ker L_s).
		\)
		Then
		\(
		W=\bigoplus_{j=0}^{m_s-1}e_s^jC_{s,W}.
		\)
		Equivalently, $W$ is free as a module over
		$\K[e_s]/(e_s^{m_s})$, with $C_{s,W}$ as multiplicity space.
		Moreover,
		\[
		\mathcal H_W(z)=(m_s)_z\,\mathcal H_{C_{s,W}}(z),
		\qquad
		m_s\mid\dim W.
		\]
	\end{corollary}
	
	\begin{proof}
		Restrict $L_s$ and $D_s$ to $W$ and apply
		Theorem~\ref{thm:truncated-q-taylor}.
	\end{proof}
	
	For an arbitrary subset $Y\subseteq X\setminus\{s\}$, put
	\[
	K_Y=\bigcap_{x\in Y}\ker\partial_x,
	\]
	with the convention $K_\varnothing=\frB$. Related kernel-stability properties occur in Lochmann's modified-shift construction
	\cite[Proposition~9(3)]{Lochmann}. Applying the truncated $q$-Weyl
	factorization to these kernel intersections gives the following.
	
	\begin{corollary}[Kernel factorization]
		\label{cor:kernel-q-weyl-factorization}
		For every $Y\subseteq X\setminus\{s\}$,
		\[
		K_Y=\bigoplus_{j=0}^{m_s-1}e_s^jC_{s,Y},
		\qquad
		C_{s,Y}=D_s^{m_s-1}(K_Y\cap\ker L_s).
		\]
		Consequently,
		\[
		\mathcal H_{K_Y}(z)=(m_s)_z\,\mathcal H_{C_{s,Y}}(z).
		\]
	\end{corollary}
	
	\begin{proof}
		For $x\neq s$, the braided-derivative identities give
		\[
		\partial_x(e_sv)=e_s\partial_x(v),
		\qquad
		\partial_xD_s=D_s\partial_x.
		\]
		Hence $K_Y$ is stable under both $L_s$ and $D_s$, and
		Corollary~\ref{cor:stable-q-weyl-factorization} applies.
	\end{proof}
	
	We now combine the kernel factorization with Gra\~na's freeness
	theorem. In the common-order setting, the resulting Hilbert-series
	divisibility recovers \cite[Theorem~5]{Lochmann}. The point here is
	that the truncated $q$-Weyl argument yields an explicit three-factor
	decomposition, with the factor $(m_s)_z$ determined locally by the
	chosen element $s\in X\setminus X'$.
	
	\begin{theorem}[Local Hilbert factorization]
		\label{thm:local-height-divisibility}
		Let $X'\subsetneq X$ be a nonempty subquandle, set
		$V'=\bigoplus_{x\in X'}\K e_x$ and $\frB'=\BN(V')$, and choose
		$s\in X\setminus X'$. Then there exists a graded vector space
		$C_{s,X'}$ such that
		\begin{equation}
			\label{eq:three-factor-local}
			\frB\simeq
			\K[e_s]/(e_s^{m_s})\otimes C_{s,X'}\otimes\frB'
		\end{equation}
		as graded vector spaces. One may take
		\[
		C_{s,X'}
		=
		D_s^{m_s-1}(K_{X'}\cap\ker L_s),
		\qquad
		K_{X'}=\bigcap_{x\in X'}\ker\partial_x.
		\]
		Consequently,
		\begin{equation}
			\label{eq:local-exact-Hilbert}
			\mathcal H_{\frB}(z)
			=
			(m_s)_z\,
			\mathcal H_{C_{s,X'}}(z)\,
			\mathcal H_{\frB'}(z).
		\end{equation}
		In particular,
		\[
		(m_s)_z\,\mathcal H_{\frB'}(z)
		\mid
		\mathcal H_{\frB}(z)
		\quad\text{in }\Z[z],
		\qquad
		m_s\dim\frB'\mid\dim\frB.
		\]
	\end{theorem}
	
	\begin{proof}
		By Gra\~na's freeness theorem,
		\(
		\frB\simeq K_{X'}\otimes\frB'
		\)
		as graded vector spaces. Since $s\notin X'$,
		Corollary~\ref{cor:kernel-q-weyl-factorization} gives
		\(
		K_{X'}
		=
		\bigoplus_{j=0}^{m_s-1}e_s^jC_{s,X'}.
		\)
		Combining the two decompositions gives
		\eqref{eq:three-factor-local} and
		\eqref{eq:local-exact-Hilbert}.
	\end{proof}
	
	Several consequences will be useful below. In the common-order
	indecomposable setting, the first one recovers
	\cite[Corollary~2]{Lochmann}; the formulation below allows the two
	local heights $m_s$ and $m_t$ to differ and does not require
	indecomposability.
	
	\begin{corollary}[Two local heights]
		\label{cor:two-local-heights}
		For distinct $s,t\in X$, there exists $Q_{s,t}(z)\in\Z_{\geq0}[z]$ such
		that $\mathcal H_{\frB}(z)=(m_s)_z(m_t)_zQ_{s,t}(z)$. In particular
		$m_sm_t\mid\dim\frB$.
	\end{corollary}
	
	\begin{proof}
		Apply Theorem~\ref{thm:local-height-divisibility} to the singleton
		subquandle $X'=\{t\}$. By Proposition~\ref{pro:local-q-weyl},
		$\BN(\K e_t)\simeq\K[e_t]/(e_t^{m_t})$, so
		$\mathcal H_{\BN(\K e_t)}(z)=(m_t)_z$.
	\end{proof}
	
	\begin{corollary}[A modular local obstruction]
		\label{cor:modular-local-obstruction}
		Suppose $\Char\K=p>0$. Let $X'\subsetneq X$ be a nonempty subquandle
		and suppose $s\in X\setminus X'$ satisfies $q_{s,s}=1$. Then
		$(p)_z\,\mathcal H_{\BN(V')}(z)\mid\mathcal H_{\frB}(z)$, and hence
		$p\,\dim\BN(V')\mid\dim\frB$.
	\end{corollary}
	
	\begin{proof}
		By Proposition~\ref{pro:local-q-weyl}, $m_s$ is the first positive
		integer for which $[m_s]_{q_s}=0$. If $q_s=1$, then $[j]_1=j$ in $\K$,
		so $m_s=p$. The result follows from
		Theorem~\ref{thm:local-height-divisibility}.
	\end{proof}
	
	\begin{corollary}[Two diagonal-one generators]
		\label{cor:modular-two-generators}
		Suppose $\Char\K=p>0$. If there exist distinct $s,t\in X$ with
		$q_{s,s}=q_{t,t}=1$, then $(p)_z^2\mid\mathcal H_{\frB}(z)$ and
		$p^2\mid\dim\frB$.
	\end{corollary}
	
	\begin{proof}
		By Proposition~\ref{pro:local-q-weyl}, $m_s=m_t=p$. Now apply
		Corollary~\ref{cor:two-local-heights}.
	\end{proof}
	
	\begin{remark}[The Weyl trace obstruction]
		\label{rmk:weyl-trace}
		For the dimension divisibility needed in the square-free application,
		the full graded factorization is stronger than necessary. Indeed,
		suppose $\Char\K=p>0$ and $q_{s,s}=1$. On every finite-dimensional
		subspace $W\subseteq\frB$ stable under $L_s$ and $D_s$,
		Proposition~\ref{pro:local-q-weyl} gives
		\[
		[D_s,L_s]=\id_W.
		\]
		The classical Weyl-algebra trace obstruction then gives
		\[
		0=\tr(D_sL_s-L_sD_s)=\tr(\id_W)=\dim W
		\qquad\text{in }\K,
		\]
		and therefore $p\mid\dim W$. Thus the dimension assertion in
		Corollary~\ref{cor:modular-local-obstruction} also follows directly:
		applying the trace obstruction to the Gra\~na kernel $K_{X'}$ gives
		$p\mid\dim K_{X'}$, while Gra\~na's freeness theorem gives
		\(
		\dim\frB=\dim K_{X'}\,\dim\BN(V').
		\)
		Hence $p\dim\BN(V')\mid\dim\frB$.
	\end{remark}

	\section{The square-free reduction}
	\label{sec:reduction}
	
	Throughout this section, $\K$ is algebraically closed with  $\Char\K=p>0$ and  $H$ is a pointed
	Hopf algebra of square-free dimension. Write
	\[
	G=\G(H),\qquad
	\gr H=R\#\K[G],\qquad
	D=\dim R=\frac{\dim H}{|G|}.
	\]
	We show that, unless $H=\K[G]$, the first layer $R(1)$ is forced to be
	one-dimensional with trivial braiding; this will ultimately force
	$D=p$. The argument proceeds by successively restricting the possible
	structure of $R(1)$.
	
	We begin with a well-known divisibility observation.
	
	\begin{lemma}
		\label{lem:nichols-divides-diagram}
		Let $U\subseteq R(1)$ be a Yetter--Drinfeld submodule. Then
		\[
		\dim\BN(U)\mid D.
		\]
	\end{lemma}
	
	\begin{proof}
		Set $V=R(1)$. Gra\~na's freeness theorem
		\cite[Theorem~3.8]{Grana} gives
		\[
		\dim\BN(U)\mid\dim\BN(V).
		\]
		Since $\dim\BN(V)\mid D$, the result follows.
	\end{proof}
	
	For the remainder of the reduction, assume $R\neq\K$. Since
	$\dim H=|G|D$ is square-free, both $D$ and $|G|$ are square-free and $(D,|G|)=1.$

	\begin{lemma}
		\label{lem:square-free-simple-summand}
		Let $U=M(g,\rho)\subseteq R(1)$ be a simple Yetter--Drinfeld
		submodule, and write $\rho(g)=q\,\id$. Then $q=1$ and $	p\mid\dim\BN(U)\mid D.$
		
	\end{lemma}
	
	\begin{proof}
		Choose $0\neq x\in U_g$ and a complement $U_g=\K x\oplus W_g$. Set
		\[
		L=\K x,\qquad
		C=W_g\oplus\bigoplus_{h\neq g}U_h,
		\]
		so that $U=L\oplus C$. Both $L$ and $C$ are $\K[G]$-subcomodules.
		The support of $L$ is $g$, hence generates $\langle g\rangle$, and both $L$ and
		$C$ are stable under this subgroup; indeed, $g\cdot x=qx$, $g$ acts
		on $U_g$ as the scalar $q$, and $g\cdot U_h=U_{ghg^{-1}}$.
		
		By Lemma~\ref{lem:nichols-divides-diagram}, $\BN(U)$ is
		finite-dimensional and $\dim\BN(U)\mid D$. Gra\~na's freeness
		theorem \cite[Theorem~3.8]{Grana} therefore applies to
		$U=L\oplus C$, and gives
		\[
		\dim\BN(L)\mid\dim\BN(U)\mid D.
		\]
		
		Since $c(x\otimes x)=q\,x\otimes x$, suppose $q\neq1$, and let
		$\ell=\ord(q)>1$. Since $q^{\ord(g)}=1$, we have
		$\ell\mid\ord(g)\mid|G|$. Also $p\nmid\ell$, since $\K^\times$ has
		no nontrivial $p$-power torsion. For the braided line $L$, the
		quantum symmetrizer in degree $n$ is multiplication by
		\[
		[n]_q!=\prod_{j=1}^n[j]_q.
		\]
		Since $[j]_q\neq0$ for $1\leq j<\ell$ and $[\ell]_q=0$,
		\[
		\BN(L)\cong\K[x]/(x^\ell),
		\qquad
		\dim\BN(L)=\ell.
		\]
		Thus $\ell\mid D$ and $\ell\mid|G|$, contradicting $(D,|G|)=1$.
		Hence $q=1$.
		
		By Remark~\ref{rmk-1-1-1-1-1}, $\BN(L)\cong\K[x]/(x^p)$. Therefore
		\[
		p=\dim\BN(L)\mid\dim\BN(U)\mid D,
		\]
		as claimed.
	\end{proof}
	
	\begin{lemma}
		\label{lem:square-free-R1-simple}
		$R(1)$ is a simple object in ${}_G^G\mathcal{YD}$.
	\end{lemma}
	
	\begin{proof}
		By Lemma~\ref{lem:square-free-simple-summand}, any simple
		Yetter--Drinfeld submodule of $R(1)$ forces $p\mid D$. Since
		$(D,|G|)=1$, we have $p\nmid|G|$. By Maschke's theorem, ${}_G^G\mathcal{YD}$ is semisimple.
		
		Write
		\[
		R(1)=U_1\oplus\cdots\oplus U_r,
		\qquad
		U_i=M(g_i,\rho_i),
		\]
		with each $U_i$ simple. By
		Lemma~\ref{lem:square-free-simple-summand}, $\rho_i(g_i)=\id$.
		Choose $0\neq x_i\in(U_i)_{g_i}$. Then $\dim\BN(\K x_i)=p$.
		Moreover, Lemma~\ref{lem:nichols-divides-diagram} gives
		\[
		\dim\BN(R(1))\mid D.
		\]
		By Remark~\ref{rmk:grana-product-divisibility},
		\[
		p^r
		=\prod_{i=1}^r\dim\BN(\K x_i)
		\mid\dim\BN(R(1))
		\mid D.
		\]
		Since $D$ is square-free, necessarily $r=1$.
		Thus $R(1)$ is simple.
	\end{proof}
	
	\begin{lemma}
		\label{lem:square-free-simple-collapse}
		$\dim R(1)=1$, and the braiding on $R(1)$ is the ordinary
		flip.
	\end{lemma}
	
	\begin{proof}
		By Lemma~\ref{lem:square-free-R1-simple}, write
		$R(1)=M(g,\rho)$ and $d=\dim\rho$.
		Lemma~\ref{lem:square-free-simple-summand} gives $\rho(g)=\id$.
		Let $W=R(1)_g$, so $\dim W=d$. For $u,v\in W$, we have
		$c(u\otimes v)=g\cdot v\otimes u=v\otimes u$, so $W$ has the
		ordinary flip braiding. If $x_1,\ldots,x_d$ is a basis of $W$, then
		\[
		\BN(W)\cong\K[x_1,\ldots,x_d]/(x_1^p,\ldots,x_d^p),
		\qquad
		\dim\BN(W)=p^d.
		\]
		The decomposition
		$R(1)=W\oplus\bigoplus_{h\neq g}R(1)_h$
		satisfies the hypotheses of Gra\~na's freeness theorem for
		$\langle g\rangle$. Therefore
		$p^d=\dim\BN(W)\mid\dim\BN(R(1))\mid D$.
		Since $D$ is square-free, $d=1$.
		
		Thus $R(1)=M(g,\chi)$ for a one-dimensional character
		$\chi\colon C_G(g)\to\K^\times$ with $\chi(g)=1$. Suppose that
		$g\notin Z(G)$. Then $|\mathcal O_g|>1$, and $R(1)$ is of quandle
		type with one-dimensional homogeneous components indexed by
		$\mathcal O_g$, with $q_{x,x}=1$ for all $x\in\mathcal O_g$. By
		Lemma~\ref{lem:nichols-divides-diagram},
		$\dim\BN(R(1))\mid D$. Choosing two distinct elements of
		$\mathcal O_g$ and applying
		Corollary~\ref{cor:modular-two-generators} gives
		$p^2\mid\dim\BN(R(1))\mid D$, contradicting the square-freeness
		of $D$. Hence $g\in Z(G)$, so $\mathcal O_g=\{g\}$. Since
		$\dim\chi=1$, we have $\dim R(1)=1$. If $v$ spans $R(1)$, then
		$c(v\otimes v)=\chi(g)v\otimes v=v\otimes v$.
	\end{proof}
	
	\begin{proposition}[Square-free reduction]
		\label{pro:square-free-reduction}
		If $R\neq\K$, then
		\[
		p\mid D,\qquad
		p\nmid|G|,\qquad
		\dim R(1)=1.
		\]
		Moreover, the braiding on $R(1)$ is the ordinary flip.
	\end{proposition}
	
	\begin{proof}
		Since $R$ is connected and coradically graded,
		$\Pp(R)=R(1)$. As $R\neq\K$, the least positive degree occurring
		in $R$ contains a nonzero primitive element, and hence
		$R(1)\neq0$.
		
		Choose a simple Yetter--Drinfeld submodule of $R(1)$.
		Lemma~\ref{lem:square-free-simple-summand} gives $p\mid D$.
		Since $(D,|G|)=1$, we have $p\nmid|G|$.
		Finally, Lemma~\ref{lem:square-free-simple-collapse} gives
		$\dim R(1)=1$ and shows that its braiding is the ordinary flip.
	\end{proof}
	
	\section{The structure theorem and classification}
	\label{sec:main}
	
	We now prove the structure theorem and complete the classification by
	combining the square-free reduction with the classification in the case
	$\dim R=p$.

	\begin{lemma}
		\label{lem:rank-one-p-power}
		Let $\K$ be algebraically closed with  $\Char\K=p>0$ and $R=\bigoplus_{n\geq0}R(n)$ a
		finite-dimensional connected coradically graded braided Hopf algebra.
		Assume that $\dim R(1)=1$ and that the braiding of $R(1)$ is trivial.
		Then $\dim R=p^a$ for some $a\geq1$.
	\end{lemma}
	
	\begin{proof}
		Set $A=R^{\operatorname{gr},*}$. Since $R$ is connected and coradically
		graded, $\Pp(R)=R(1)$. By the graded-duality principle
		\cite[Lemma~2.4]{AS02}, the graded algebra $A$ is generated by $A(1)$.
		Since $\dim A(1)=1$, choosing $0\neq z\in A(1)$ gives
		$A\cong\K[z]/(z^N)$ for some $N\geq2$, and hence
		$N=\dim A=\dim R$.
		
		As $A$ is connected and graded, $z$ is primitive. 
		Moreover, triviality of the braiding on $R(1)$ implies
		$c_A(z\otimes z)=z\otimes z$. Thus $z\otimes1$ and $1\otimes z$
		commute in the braided tensor-product algebra, and therefore
		\[
		0=\Delta(z^N)=\sum_{j=0}^N\binom Nj z^j\otimes z^{N-j}.
		\]
		By the minimality of $N$, both factors in $z^j\otimes z^{N-j}$ are
		nonzero for $0<j<N$. Since these terms have distinct bidegrees, it
		follows that $\binom Nj=0$ in $\K$ for $0<j<N$. The $q$-binomial
		criterion of Radford \cite{Radford99}, specialized to
		$q=1$, now yields $N=p^a$ for some $a\geq0$. Since $N\geq2$,
		necessarily $a\geq1$. Thus $\dim R=p^a$.
	\end{proof}

	\begin{theorem}
		\label{thm:square-free-pointed}
		Let $\K$ be algebraically closed with  $\Char\K=p>0$ and  $H$  a pointed Hopf algebra
		of square-free dimension. Then either $H=\K[\G(H)]$
		or $\dim H=p|\G(H)|$.
		In particular, if $p\nmid\dim H$, then $H$ is a group algebra.
	\end{theorem}
	
	\begin{proof}
		Put $G=\G(H)$, $\gr H=R\#\K[G]$, and $D=\dim R$.
		If $R=\K$, then $\dim H=|G|=\dim H_0$, so
		$H=H_0=\K[G]$.
		If $R\neq\K$, Proposition~\ref{pro:square-free-reduction} gives
		$\dim R(1)=1$ with trivial braiding.
		Lemma~\ref{lem:rank-one-p-power} yields $D=p^a$.
		Since $D>1$ and $D$ is square-free, $a=1$ and $D=p$.
		Therefore $\dim H/|\G(H)|=p$. This alternative makes
		$\dim H=p|G|$ divisible by $p$, so $p\nmid\dim H$
		forces $H=\K[G]$.
	\end{proof}
	
	The preceding theorem reduces the non-group case to pointed
	Hopf algebras whose diagram has dimension $p$. We now complete the
	proof of the classification stated in
	Theorem~\ref{thm:intro-classification}.
	
	\begin{proof}[The proof of Theorem~\ref{thm:intro-classification}]
		Let $G=\G(H)$ and $\gr H=R\#\K[G]$.
		Theorem~\ref{thm:square-free-pointed} gives either $R=\K$,
		in which case $H=\K[G]$, or $\dim R=p$.
		In the latter case $n=pm$, $|G|=m$, and $p\nmid|G|$,
		since $n$ is square-free.
		By \cite[Theorem~3.7]{Xiong2023}, $H$ is isomorphic to one of
		$\fA^1_G(g,\chi,f)$, $\fA^2_G(g,\chi,f)$, or
		$\fA^3_G(g,f)$.
		By \cite[Remark~3.2(3)]{Xiong2023}, one may choose the generator
		with $f=0$ because $p\nmid|G|$. The third family requires
		$p\mid\ord(g)$ by \cite[Definition~3.1]{Xiong2023}.
		Since $\ord(g)\mid|G|$ and $p\nmid|G|$, this family cannot occur.
		
		Proposition~\ref{pro:square-free-reduction} gives
		$\dim R(1)=1$ with ordinary flip braiding. Hence
		\[
		R(1)=M(g,\chi)
		\]
		for some $g\in Z(G)$ and
		$\chi\in\Hom(G,\K^\times)$ with $\chi(g)=1$.
		We thus denote the two remaining normalized families by
		$\fA^1_G(g,\chi)$ and $\fA^2_G(g,\chi)$.
		With $f=0$, \cite[Definition~3.1]{Xiong2023} gives the displayed
		presentations and, for $\fA^2$, exactly the additional conditions
		$g^{p-1}=1$ and $\chi^{p-1}=\varepsilon$.
		The two families are mutually non-isomorphic by
		\cite[Proposition~3.6]{Xiong2023}.
		
		Finally, an isomorphism within either family restricts on the
		group-like elements to a group isomorphism $F:G\to G'$.
		The induced map on the one-dimensional infinitesimal braidings
		preserves the coaction and action, and hence
		\[
		F(g)=g',
		\qquad
		\chi'=\chi\circ F^{-1}.
		\]
		
		Conversely, any group isomorphism $F:G\to G'$ satisfying these
		conditions extends to a Hopf algebra isomorphism by
		\[
		h\longmapsto F(h)\quad(h\in G),
		\qquad
		x\longmapsto x',
		\]
		as is immediate from the displayed relations and coproducts.
	\end{proof}

	\section*{Acknowledgments}
	
	This work was partially supported by the National Natural Science Foundation of China (Grant No.~12401041). The author would like to thank Dr. Huan Jia for reading an earlier draft of this paper and for his helpful comments and suggestions.

%
%
%
%
%

\end{document}